\documentclass[10pt,reqno]{amsart}
\usepackage{amsmath,amscd}

\theoremstyle{plain}
   \newtheorem{theorem}{Theorem}[section]
   \newtheorem{proposition}[theorem]{Proposition}
   \newtheorem{lemma}[theorem]{Lemma}
   \newtheorem{corollary}[theorem]{Corollary}
   \newtheorem{conjecture}[theorem]{Conjecture}
   
\theoremstyle{definition}

\theoremstyle{remark}
 \newtheorem{remark}{Remark}[section]
\newcommand{\xx}{\mathbf{x}}
\newcommand{\LL}{\mathcal{L}}

\newcommand{\yy}{\mathbf{y}}

\newcommand{\NN}{\mathbb{N}}

\newcommand{\EE}{\mathcal{E}}

\newcommand{\ZZ}{\mathbb{Z}}

\newcommand{\RR}{\mathbb{R}}
\newcommand{\CC}{\mathbb{C}}

\newcommand{\sym}{\mathfrak{S}}

\def\newop#1{\expandafter\def\csname #1\endcsname{\mathop{\rm
#1}\nolimits}}

\newop{diag}
\newop{supp}
\newop{JP}
\newop{Sym}
\newop{si}
\newop{glb}
\newop{des}
\newop{fmaj}
\newop{comaj}
\newop{maj}
\newop{lhp}

\begin{document}
\title[]
{Lecture hall $P$-partitions}

\author{Petter Br\"and\'en}
\author{Madeleine Leander}
\address{Department of Mathematics, Royal Institute of Technology, SE-100 44 Stockholm,
Sweden}
\email{pbranden@kth.se}
\address{Department of Mathematics, Stockholm University, SE-106 91
  Stockholm, Sweden}
\email{madde@math.su.se}

\begin{abstract}
We introduce and study $s$-lecture hall $P$-partitions which is a generalization of $s$-lecture hall partitions to labeled (weighted) posets. We provide generating function identities for $s$-lecture hall $P$-partitions that generalize identities obtained by Savage and Schuster for $s$-lecture hall partitions, and by Stanley for $P$-partitions. We also prove that the corresponding $(P,s)$-Eulerian polynomials are real-rooted for certain pairs $(P,s)$, and speculate on unimodality properties of these polynomials. 
\end{abstract}

\maketitle
\thispagestyle{empty}

\section{Introduction}
Let $s = (s_1, \ldots , s_n)$ be a sequence of positive integers. An $s$-\emph{lecture hall partition} is
an integer sequence $\lambda =(\lambda_1,\ldots , \lambda_n)$ satisfying $0 \leq \lambda_1/s_1 \leq \cdots \leq \lambda_n/s_n$. These are generalizations of \emph{lecture hall partitions}, corresponding to the case when $s=(1,2,\ldots,n)$, first studied by Bousquet-M\'elou and Eriksson \cite{BME}. 
It has recently been made evident that $s$-lecture hall partitions serve as a rich model for various combinatorial structures with nice generating functions, see \cite{BeBr,BME, BME2, CoLeSa, Sav, SP2, SS, SaWi} and the references therein.

In this paper we generalize the concept of $s$-lecture hall partitions to labeled posets. This constitutes a generalization of Stanley's theory of $P$-partitions, see \cite[Ch. 3.15]{EC1}. In Section \ref{gener} we derive multivariate generating function identities for $s$-lecture hall $P$-partitions, and prove a reciprocity theorem (Theorem \ref{reci}). When $P$ is a naturally labeled chain or an anti-chain, the generating function identities obtained  produce results on $s$-lecture hall partitions and signed permutations, respectively (see Section \ref{appl}). We also introduce and  study a $(P,s)$-Eulerian polynomial. In Section \ref{SR} we prove that this polynomial is palindromic for sign-graded labeled posets with a specific choice of $s$. In Section \ref{roots} we prove that the $(P,s)$-Eulerian polynomial is real-rooted for certain choices of $(P,s)$, and we also speculate on unimodality properties satisfied by these polynomials.

\section{Lecture hall $P$-partitions}
In this paper a \emph{labeled poset}  is a partially ordered set on $[p]:=\{1,\ldots, p\}$ for some positive integer  $p$, i.e., $P=([p], \preceq)$, where $\preceq$ denotes the partial order. 
We will use the symbol $\leq$ to denote the usual total order on the integers. 
If $P$ is a labeled poset, then a $P$-partition\footnote{What we call $P$-partitions are called reverse $(P,\omega)$-partitions in \cite{EC1,EC2}. However the theory of  $(P,\omega)$-partitions  and reverse $(P,\omega)$-partitions are clearly equivalent.} is a map $f : [p] \rightarrow \RR$ such that 
\begin{enumerate}
\item if $x \prec y$, then $f(x) \leq f(y)$, and  
\item if $x \prec y$ and $x>y$, then $f(x)<f(y)$.
\end{enumerate} 
The theory of $P$-partitions was developed by Stanley in his thesis and has since then been used frequently in several different combinatorial settings, see \cite{EC1, EC2}. 

Let 
$$
O(P)= \{ f \in \RR^p: f \mbox{ is a $P$-partition and } 0 \leq f(x) \leq 1 \mbox{ for all } x \in [p]\}
$$ 
be the \emph{order polytope} associated to $P$. Note that if $P$ is naturally labeled, i.e., $x \prec y$ implies $x<y$, then $O(P)$ is a closed integral polytope. Otherwise $O(P)$ is the intersection of a finite number of open or closed half-spaces. Recall that the \emph{Ehrhart polynomial} of an integral polytope $\mathcal{P}$ in $\RR^p$ is defined for nonnegative integers $n$ as 
$$
i(\mathcal{P}, n) = |n\mathcal{P} \cap \ZZ^p|, 
$$
where $n\mathcal{P}= \{ n \xx : \xx \in \mathcal{P}\}$, see \cite[p.~497]{EC1}.  
For order polytopes we have the following relationship due to Stanley: 
$$
\sum_{n\geq 0} i(O(P), n)t^n = \frac {A_{P}(t)} {(1-t)^{p+1}}, 
$$
where $A_{P}(t)$ is the $P$-Eulerian polynomial, which is the generating polynomial of the descent statistic over the set of all linear extensions of $P$, see \cite[Ch.~3.15]{EC1}.

The purpose of this paper is to initiate the study of a  lecture hall generalization of $P$-partitions. Let $P$ be a labeled poset and let $s : [p] \rightarrow \ZZ_+:=\{1,2,3,\ldots\}$ be an arbitrary map. We define a  \emph{lecture hall} $(P,s)$-\emph{partition} to be a map $f : [p] \rightarrow \RR$ such that 
\begin{enumerate}
\item if $x \prec y$, then $f(x)/s(x) \leq f(y)/s(y)$, and  
\item if $x \prec y$ and $x>y$, then $f(x)/s(x)<f(y)/s(y)$.
\end{enumerate} 
Let 
$$
O(P,s)= \{ f \in \RR^p: f \mbox{ is a $(P,s)$-partition and } 0 \leq f(x)/s(x) \leq 1 \mbox{ for all } x \in [p]\} 
$$ 
be the \emph{lecture hall order polytope} associated to $(P,s)$. We also let 
$$
C(P,s)= \{ f \in \RR^p: f \mbox{ is a $(P,s)$-partition and } 0 \leq f(x)/s(x)  \mbox{ for all } x \in [p]\} 
$$ 
be the \emph{lecture hall order cone} associated to $(P,s)$. The $(P,s)$-\emph{Eulerian polynomial}, $A_{(P,s)}(t)$, is defined by 
$$
\sum_{n \geq 0} i(O(P,s),n)t^n = \frac {A_{(P,s)}(t)}{(1-t)^{p+1}}.
$$

\section{The main generating functions}\label{gener}
In this section we derive formulas for the main generating functions associated to lecture hall $(P,s)$-partitions. The outline follows Stanley's theory of $P$-partitions \cite[Ch.~3.15]{EC1}. We shall see in Section~\ref{appl} that the special cases when $P$ is naturally labeled chain or an anti-chain automatically produce results on lecture hall polytopes and signed permutations, respectively. 

Let $\sym_p$ denote the symmetric group on $[p]$. If $\pi = \pi_1\pi_2 \cdots \pi_p \in \sym_p$ is a permutation written in one-line notation, we let $P_\pi$ denote the labeled chain $\pi_1 \prec \pi_2 \prec \cdots \prec \pi_p$. If  $P=([p], \preceq)$ is a labeled poset, let $\LL(P)$ denote the set 
$$\LL(P):=\{\pi \in \sym_p : \mbox{if } \pi_i \preceq \pi_j, \mbox{ then } i\leq j, \mbox{ for all } i,j \in [p]\}, $$
of \emph{linear extensions} (or the \emph{Jordan-H\"older set}) of $P$. 
The following lemma is an immediate consequence of Stanley's decomposition of $P$-partitions \cite[Lemma 3.15.3]{EC1}. 
\begin{lemma}\label{decomp}
If $P$ is a labeled poset and $s : [p] \rightarrow \ZZ_+$, then 
$$
C(P,s)= \bigsqcup_{\pi \in \LL(P)}\!\!C(P_\pi,s), 
$$ 
where $\bigsqcup$ denotes disjoint union. 
\end{lemma}
Let  $s : [p] \rightarrow \ZZ_+$. 
An $s$-\emph{colored permutation} is a pair $\tau=(\pi, r)$ where $\pi \in \sym_p$, and $r : [p] \rightarrow \NN$ satisfies $r(\pi_i)\in \{0,1,\ldots, s(\pi_i)-1\}$ for all $1\leq i \leq p$. 
If $P=([p], \preceq)$ is a labeled poset, let 
$$
\LL(P,s)= \{ \tau: \tau= (\pi, r) \mbox{ where } \pi \in \LL(P) \mbox{ and } \tau \mbox{ is  an $s$-colored permutation}\}.
$$
For $f : [p] \rightarrow \NN$, let $q(f), r(f) : [p] \rightarrow \NN$ be the unique functions satisfying 
$$
f(x)= q(f)(x)\cdot s(x)+ r(f)(x), \ \ \mbox{ where } q(f)(x) \in \NN \mbox{ and } 0 \leq r(f)(x)<s(x), 
$$
for all $x \in [p]$. Let further
$$
F_{(P,s)}(\xx,\yy) = \sum_{f \in \NN(P,s)} \yy^{r(f)} \xx^{q(f)}, 
$$
where $\xx^r=x_1^{r(1)}x_2^{r(2)}\cdots x_p^{r(p)}$ and $\NN(P,s)= C(P,s)\cap \NN^p$. 
 We say that $i \in [p-1]$ is a \emph{descent} of $\tau=(\pi,r)$ if 
$$
\begin{cases}
\pi_i < \pi_{i+1} \mbox{ and } r(\pi_i)/s(\pi_i) > r(\pi_{i+1})/s(\pi_{i+1}), \mbox{ or,} \\
\pi_i > \pi_{i+1} \mbox{ and } r(\pi_i)/s(\pi_i) \geq r(\pi_{i+1})/s(\pi_{i+1}), 
\end{cases}
$$
Let 
$$
D_1(\tau)= \{i \in [p-1]: i \mbox{ is a descent}\}. 
$$
\begin{theorem}\label{maingen}
If $P$ is a labeled poset and $s : [p] \rightarrow \ZZ_+$, then 
\begin{eqnarray}
\label{xy}
F_{(P,s)}(\xx,\yy) = \sum_{\tau=(\pi,r) \in \LL(P,s)} \yy^r\frac {\displaystyle \prod_{i\in D_1(\tau)}  x_{\pi_{i+1}} \cdots x_{\pi_{p}}}{\displaystyle \prod_{i\in [p]}(1-x_{\pi_i}\cdots x_{\pi_p})}.
\end{eqnarray}
\end{theorem} 

\begin{proof}
By Lemma \ref{decomp} we may assume that  $P= P_\pi$ is a labeled chain.  Let $f \in \NN^p$, and write 
$f(t)= q(t)s(t)+r(t)$, where $0 \leq r(t) <s(t)$ and $q(t)\in \NN$ for all $t \in [p]$. What conditions on $q$ and $r$ guarantee $f \in \NN(P,s)$? Suppose $\pi_i < \pi_{i+1}$. Then we need 
\begin{equation}\label{frac}
q(\pi_i) + \frac {r(\pi_i)}{s(\pi_i)} =\frac{f(\pi_i)}{s(\pi_i)} \leq \frac {f(\pi_{i+1})}{s(\pi_{i+1})}
= q(\pi_{i+1}) + \frac {r(\pi_{i+1})}{s(\pi_{i+1})}. 
\end{equation}
If $r(\pi_i)/s(\pi_i) \leq r(\pi_{i+1})/s(\pi_{i+1})$, then \eqref{frac} holds if and only if $q(\pi_i) \leq q(\pi_{i+1})$. If $r(\pi_i)/s(\pi_i) > r(\pi_{i+1})/s(\pi_{i+1})$, then \eqref{frac} holds if and only if $q(\pi_i) < q(\pi_{i+1})$.

Suppose $\pi_i > \pi_{i+1}$. Then we need 
\begin{equation}\label{frac2}
q(\pi_i) + \frac {r(\pi_i)}{s(\pi_i)} =\frac{f(\pi_i)}{s(\pi_i)} < \frac {f(\pi_{i+1})}{s(\pi_{i+1})}
= q(\pi_{i+1}) + \frac {r(\pi_{i+1})}{s(\pi_{i+1})}. 
\end{equation}
If $r(\pi_i)/s(\pi_i) < r(\pi_{i+1})/s(\pi_{i+1})$, then \eqref{frac2} holds if and only if $q(\pi_i) \leq q(\pi_{i+1})$. If $r(\pi_i)/s(\pi_i) \geq r(\pi_{i+1})/s(\pi_{i+1})$, then \eqref{frac2} holds if and only if $q(\pi_i) < q(\pi_{i+1})$. 

Let $\tau =(\pi, r)$, where $r$ is fixed. Then $f=qs+r \in \NN(P,s)$ with given (fixed) $r$ if and only if 
\begin{equation}\label{daqs}
0 \leq q(\pi_1) \leq q(\pi_2) \leq \cdots \leq q(\pi_p),
\end{equation}
where $q(\pi_i)<q(\pi_{i+1})$ if $i \in D_1(\tau)$. Hence $f=qs+r \in \NN(P,s)$ if and only if for each $k \in [p]$:
$$q(\pi_k)= \alpha_k + |\{i \in D_1(\tau): i<k\}|,$$ where $\alpha_k \in \NN$ and $0\leq \alpha_1 \leq \cdots \leq \alpha_p$. Hence 
\begin{align*}
\sum_q \prod_{i=1}^p x_{\pi_{i}}^{q(\pi_i)} &= \sum_{0\leq \alpha_1\leq \cdots \leq \alpha_p} x_{\pi_1}^{\alpha_1} \cdots  x_{\pi_p}^{\alpha_p}  \prod_{i\in D_1(\tau)}  x_{\pi_{i+1}} \cdots x_{\pi_{p}} \\
&=\frac {\displaystyle \prod_{i\in D_1(\tau)}  x_{\pi_{i+1}} \cdots x_{\pi_{p}}}{\displaystyle \prod_{i\in [p]}(1-x_{\pi_i}\cdots x_{\pi_p})}, 
\end{align*}
where the first sum is over all $q$ satisfying \eqref{daqs}. The theorem follows.
\end{proof}

Let $\ZZ_+(P,s)=C(P,s) \cap \ZZ_+^p$ and let 
$$
F_{(P,s)}^+(\xx,\yy) = \sum_{f \in \ZZ_+(P,s)} \yy^{r(f)} \xx^{q(f)}.
$$
Let further 
$$
D_2(\tau)=
\begin{cases}
D_1(\tau), &\mbox{ if } r(\pi_1) \neq 0, \\
D_1(\tau) \cup\{0\},  &\mbox{ if } r(\pi_1) = 0.
\end{cases}
$$
\begin{theorem}
If $P$ is a labeled poset and $s : [p] \rightarrow \ZZ_+$, then 
$$
F_{(P,s)}^+(\xx,\yy) = \sum_{\tau=(\pi,r) \in \LL(P,s)} \yy^r\frac {\displaystyle \prod_{i\in D_2(\tau)}  x_{\pi_{i+1}} \cdots x_{\pi_{p}}}{\displaystyle \prod_{i\in [p]}(1-x_{\pi_i}\cdots x_{\pi_p})}.
$$
\end{theorem} 

\begin{proof}
Consider $(P', s')$ where $P'$ is obtained from $P$ by adjoining a least element $\hat{0}$ labeled $p+1$, and $s' : [p+1] \rightarrow \ZZ_+$ is such that $s'$ restricted to $[p]$ agrees with $s$. Let also $s'(p+1)> \max\{s(t) : t \in [p]\}$. Then  $f \in \NN(P',s')$ if and only if 
$f|_{[p]} \in  \NN(P,s)$ and 
$$
0\leq \frac {f(p+1)}{s'(p+1)} < \frac {f(x)}{s(x)}, \ \ \mbox{ for all }  x \in [p]. 
$$
Thus $F_{(P,s)}^+(\xx,\yy)$ is obtained from $F_{(P',s')}(\xx,\yy)$ when we restrict to all $f \in \NN(P',s')$ 
with $f(p+1)=1$, i.e., $q(p+1)=0$ and $r(p+1)=1$, and then shift the indices. Hence $i=0$ is a descent in 
$((p+1)\pi_1\pi_2 \cdots \pi_p, r)$ if and only if $r(\pi_1)=0$, and the proof follows. 
\end{proof}
For $f : [p] \rightarrow \ZZ_+$, let $q'(f), r'(f) : [p] \rightarrow \NN$ be the unique functions satisfying 
$$
f(x)= q'(f)(x)\cdot s(x)+ r'(f)(x), \ \ \mbox{ where } q'(f)(x) \in \NN \mbox{ and } 0 < r'(f)(x)\leq s(x), 
$$
for all $x \in [p]$. Let further
$$
G_{(P,s)}(\xx,\yy) = \sum_{f \in \ZZ_+(P,s)} \yy^{r'(f)} \xx^{q'(f)}.
$$
Let  $D_3(\tau)$ be the set of all $i \in [p-1]$ for which
\begin{align*}
&\pi_i < \pi_{i+1} \mbox{ and } (r(\pi_i)+1)/s(\pi_i) > (r(\pi_{i+1})+1)/s(\pi_{i+1}), \mbox{ or,} \\
&\pi_i > \pi_{i+1} \mbox{ and } (r(\pi_i)+1)/s(\pi_i) \geq (r(\pi_{i+1})+1)/s(\pi_{i+1}). 
\end{align*}

\begin{theorem}
If $P$ is a labeled poset and $s : [p] \rightarrow \ZZ_+$, then 
$$
G_{(P,s)}(\xx,\yy) = \sum_{\tau=(\pi,r) \in \LL(P,s)} \yy^{r+\mathbf{1}}\frac {\displaystyle \prod_{i\in D_3(\tau)}  x_{\pi_{i+1}} \cdots x_{\pi_{p}}}{\displaystyle \prod_{i\in [p]}(1-x_{\pi_i}\cdots x_{\pi_p})}, 
$$
where $\mathbf{1}=(1,1,\ldots,1)$ is the all ones vector.
\end{theorem} 
\begin{proof}
The proof is almost identical to that of Theorem \ref{maingen}, and is therefore omitted. 
\end{proof}

For $n \in \NN$, let $$\NN_{\leq n}(P,s)=\{f \in \NN(P,s): f(x)/s(x) \leq n \mbox{ for all } x \in [p]\},$$ and let
$$
F_{(P,s)}(\xx,\yy;n) = \sum_{f \in \NN_{\leq n}(P,s)} \yy^{r(f)} \xx^{q(f)}.
$$
The polynomials $F_{(P,s)}^+(\xx,\yy;n)$ and $G_{(P,s)}(\xx,\yy;n)$ are defined analogously over $\{f \in \ZZ_+(P,s): f(x)/s(x) \leq n \mbox{ for all } x \in [p]\}$. Let also $$\NN_{<n}(P,s)=\{f \in \NN(P,s): f(x)/s(x) < n \mbox{ for all } x \in [p]\},$$ and 
$$
F'_{(P,s)}(\xx,\yy;n) = \sum_{f \in \NN_{<n}(P,s)} \yy^{r(f)} \xx^{q(f)}.
$$
For $\tau=(\pi,r) \in \LL(P,s)$, define  
$$
D(\tau)=
\begin{cases}
D_1(\tau), &\mbox{ if } r(\pi_p) = 0, \\
D_1(\tau) \cup\{p\},  &\mbox{ if } r(\pi_p) > 0, 
\end{cases}
$$
and 
$$
D_4(\tau)=
\begin{cases}
D_2(\tau), &\mbox{ if } r(\pi_p) = 0, \\
D_2(\tau) \cup\{p\},  &\mbox{ if } r(\pi_p) > 0.
\end{cases}
$$

\begin{proposition}\label{eqs}
If $P$ is a labeled poset and $s : [p] \rightarrow \ZZ_+$, then 
\begin{align}
\sum_{n \geq 0} F_{(P,s)}(\xx,\yy;n)t^n &=  \sum_{\tau=(\pi,r) \in \LL(P,s)} \yy^r \frac{\displaystyle \prod_{i\in D(\tau)}  x_{\pi_{i+1}} \cdots x_{\pi_{p}}}{\displaystyle \prod_{i\in [p]}(1-x_{\pi_i}\cdots x_{\pi_p}t)}\frac {t^{|D(\tau)|}} {1-t} \label{r1},\\
\sum_{n \geq 0} F'_{(P,s)}(\xx,\yy;n)t^n &=  \sum_{\tau=(\pi,r) \in \LL(P,s)} \yy^r \frac{\displaystyle \prod_{i\in D_1(\tau)}  x_{\pi_{i+1}} \cdots x_{\pi_{p}}}{\displaystyle \prod_{i\in [p]}(1-x_{\pi_i}\cdots x_{\pi_p}t)}\frac {t^{|D_1(\tau)|+1}} {1-t} \label{r2},\\
\sum_{n \geq 0} F_{(P,s)}^+(\xx,\yy;n)t^n &=  \sum_{\tau=(\pi,r) \in \LL(P,s)} \yy^r \frac{\displaystyle \prod_{i\in D_4(\tau)}  x_{\pi_{i+1}} \cdots x_{\pi_{p}}}{\displaystyle \prod_{i\in [p]}(1-x_{\pi_i}\cdots x_{\pi_p}t)}\frac {t^{|D_4(\tau)|}} {1-t} \label{r3}, \\
\sum_{n \geq 0} G_{(P,s)}(\xx,\yy;n)t^n &=  \sum_{\tau=(\pi,r) \in \LL(P,s)} \yy^{r+\mathbf{1}} \frac{\displaystyle \prod_{i\in D_3(\tau)}  x_{\pi_{i+1}} \cdots x_{\pi_{p}}}{\displaystyle \prod_{i\in [p]}(1-x_{\pi_i}\cdots x_{\pi_p}t)}\frac {t^{|D_3(\tau)|+1}} {1-t} \label{r4}. 
\end{align}
\end{proposition}
\begin{proof}
For \eqref{r1} consider $(P', s')$ where $P'$ is obtained from $P$ by adjoining a greatest element $\hat{1}$ labeled $p+1$, and $s' : [p+1] \rightarrow \ZZ_+$ restricted to $[p]$ agrees with $s$, while $s'(p+1)=1$. If we set $x_{p+1}=t$, then 
$$
\sum_{n \geq 0} F_{(P,s)}(\xx,\yy;n)t^n = F_{(P',s')},
$$
and 
$$
\LL(P',s')= \{(\pi_1\cdots \pi_p (p+1), r') : (\pi_1\cdots \pi_p, r'|_P)\in \LL(P,s) \mbox{ and } r'(p+1)=0\}.
$$
The identity \eqref{r1} follows by noting that $i=p$ is a descent of $(\pi_1\cdots \pi_p (p+1), r')$ if and only if $r(\pi_p)/s(\pi_p) > r'(p+1)/s'(p+1)=0$.

The other identities follows similarly. 
For example \eqref{r2} follows by considering $(P', s')$ where $P'$ is obtained from $P$ by adjoining a greatest element $\hat{1}$ labeled $0$ (and then relabel so that $P'$ has ground set $[p+1]$). 
For \eqref{r4} consider again $(P', s')$, where $P'$ is obtained from $P$ by adjoining a greatest element $\hat{1}$ labeled $p+1$, and $s'$ is defined as for the case of \eqref{r1}. Note that since $r'(p+1)=1$ we have $q'(p+1)=n-1$ if $f(p+1)=n$. This explains the shift by one in the exponent on the right hand side of \eqref{r4}, i.e., $|D_3(\tau)|+1$. 
\end{proof}

If $q$ is a variable, let $[0]_q:=0$ and $[n]_q:= 1+q+q^2+\cdots + q^{n-1}$ for $n \geq 1$. 
For the special case of \eqref{r1} when $P$ is an anti-chain we acquire the following corollary, which is a generalization of \cite[Theorem 5.23]{BeckBraun}. 
\begin{corollary} 
\label{6}
If $P$ is an anti-chain and $s : [p] \rightarrow \ZZ_+$, then 
\begin{eqnarray*}
 \sum_{n \geq 0} \prod_{i=1}^p   \left(  x_i^n + [n]_{x_i}[s(i)]_{y_i}     \right)t^n =\sum_{\tau=(\pi,r) \in \LL(P,s)} \yy^r \frac{\displaystyle \prod_{i\in D(\tau)}  x_{\pi_{i+1}} \cdots x_{\pi_{p}}}{\displaystyle \prod_{i\in [p]}(1-x_{\pi_i}\cdots x_{\pi_p}t)}\frac {t^{|D(\tau)|}} {1-t}
\end{eqnarray*}
\end{corollary}
\begin{proof}
 Let $P$ be an anti-chain and let $s:[p]\rightarrow \ZZ_+ $. Consider $f \in \NN_{\leq n}(P,s)$. Since $P$ is an anti-chain, $f(i)$ and $f(j)$ are independent for all $1\leq i < j \leq p$, and the only restriction is $0 \leq f(i)\leq ns(i)$ for all $1\leq i \leq p$. 
We write $f(i)=s(i) q(i)+ r(i)$, where $0\leq r(i)< s(i)$. Then $f \in \NN_{\leq n}(P,s)$ if and only if either $q(i)=n$ and $r(i)=0$, or $0 \leq q(i) \leq n-1$ and $0 \leq r(i) \leq s(i)-1$. Hence 
\begin{eqnarray*}
\sum_{f \in \NN_{\leq n}(P,s)} \yy^{r(f)} \xx^{q(f)} &=& \prod_{i=1}^p \left( x_i^0 [s(i)]_{y_i}+ \cdots + x_i^{n-1}[s(i)]_{y_i} + x_i^n \right) \\
&=& \prod_{i=1}^p   \left(   x_i^n + [n]_{x_i}[s(i)]_{y_i}\right). 
\end{eqnarray*}
The corollary now follows from \eqref{r1}. 
\end{proof}

Note that the special case of \eqref{r1} when $P$ is a naturally labeled chain gives an analogue (by an appropriate change of variables)  to one of the main results in \cite{SS}, see Theorem 5 therein. 
From \eqref{r1} we also get an interpretation of the Eulerian polynomial $A_{(P,s)}(t)$. For $\tau \in \LL(P,s)$, let $\des_s(\tau)= |D(\tau)|.$
\begin{corollary}\label{eul}
If $P$ is a labeled poset and $s : [p] \rightarrow \ZZ_+$, then 
$$
A_{(P,s)}(t) = \sum_{\tau \in \LL(P,s)} t^{\des_s(\tau)}. 
$$
\end{corollary}
The next corollary follows from Proposition \ref{eqs} by setting the $x$- and $y$-variables to $1$.
\begin{corollary}\label{eul2}
If $P$ is a labeled poset and $s : [p] \rightarrow \ZZ_+$, then 
$$
 \sum_{\tau \in \LL(P,s)} t^{|D_4(\tau)|} =  \sum_{\tau \in \LL(P,s)} t^{|D_3(\tau)|+1},
 $$
and if $s(x)=1$ for all minimal elements $x$ in $P$, then  
\begin{equation}\label{eul3}
A_{(P,s)}(t)= \sum_{\tau \in \LL(P,s)} t^{|D(\tau)|} =  \sum_{\tau \in \LL(P,s)} t^{|D_3(\tau)|}.
\end{equation}
 \end{corollary}

Let $P=([p], \preceq)$ be a labeled poset. For $i \in [p]$, let $i^*=p+1-i$, and let $(P^*,s^*)$ be defined by $P^*=([p], \preceq^*)$ with 
$$
i \preceq j \mbox{ in } P\ \ \mbox{ if and only if } \ \ i^* \preceq^* j^* \mbox{ in } P^*, \ \ \mbox{ for all } i,j \in [p], 
$$
and $s^*(i^*)=s(i)$ for all $i \in [p]$. The poset $P^*$ is called the \emph{dual} of $P$.

\begin{theorem}[Reciprocity theorem]\label{reci}
If $P$ is a labeled poset and $s : [p] \rightarrow \ZZ_+$, then 
$$
G_{(P^*,s^*)}(\xx^*,\yy^*) = (-1)^p \frac {y_1^{s(1)}\cdots y_p^{s(p)}}{x_1\cdots x_p} F_{(P,s)}(\xx^{-1},\yy^{-1}),$$
where $\xx^*= (x_p,x_{p-1},\ldots, x_1)$ and $\xx^{-1}=(x_1^{-1},\ldots, x_p^{-1})$. 
\end{theorem}
\begin{proof}
For $\tau =(\pi,r) \in \LL(P, s)$, let $\tau^* = (\pi_1^*\pi_2^*\cdots \pi_p^*,r^*)$ where $r^*(i^*) = s(i)-1-r(i)$ for all $i \in [p]$. Clearly the map $\tau \mapsto \tau^*$ is a bijection between $\LL(P, s)$ and $\LL(P^*, s^*)$. Moreover if $i \in [p-1]$, then $i \in D_3(\tau)$ if and only if 
$$
\begin{cases}
\pi_i < \pi_{i+1} \mbox{ and } (r(\pi_i)+1)/s(\pi_i) > (r(\pi_{i+1})+1)/s(\pi_{i+1}), \mbox{ or,} \\
\pi_i > \pi_{i+1} \mbox{ and } (r(\pi_i)+1)/s(\pi_i) \geq (r(\pi_{i+1})+1)/s(\pi_{i+1}), 
\end{cases}
$$
 if and only if
 $$
\begin{cases}
\pi_i^* > \pi_{i+1}^* \mbox{ and } r^*(\pi_i^*)/s^*(\pi_i^*) < r^*(\pi_{i+1}^*)/s^*(\pi_{i+1}^*) , \mbox{ or,} \\
\pi_i^* < \pi_{i+1}^* \mbox{ and }  r^*(\pi_i^*)/s^*(\pi_i^*) \leq r^*(\pi_{i+1}^*)/s^*(\pi_{i+1}^*)
\end{cases}
$$
if and only if $i \in [p-1] \setminus D_1(\tau^*)$. Thus 
\begin{equation}\label{comp}
D_3(\tau)= [p-1] \setminus D_1(\tau^*) \ \ \mbox{ and } \ \  D_1(\tau)= [p-1] \setminus D_3(\tau^*), 
\end{equation}
for all $\tau \in \LL(P,s)$. 
Now 
\begin{align*}
F_{(P,s)}(\xx,\yy) &= \sum_{\tau \in \LL(P,s)} \yy^r\frac {\displaystyle \prod_{i\in D_1(\tau)}  x_{\pi_{i+1}} \cdots x_{\pi_{p}}}{\displaystyle \prod_{i\in [p]}(1-x_{\pi_i}\cdots x_{\pi_p})} = \sum_{\tau \in \LL(P,s)} \yy^r\frac {\displaystyle \prod_{i\in [p-1]\setminus D_3(\tau^*)}  x_{\pi_{i+1}} \cdots x_{\pi_{p}}}{\displaystyle \prod_{i\in [p]}(1-x_{\pi_i}\cdots x_{\pi_p})}  \\
&=\sum_{\tau \in \LL(P,s)} \frac{\yy^{s}(\yy^*)^{-(r^*+\mathbf{1})}}{x_1\cdots x_p}\frac {\displaystyle \prod_{i\in D_3(\tau^*)}  x_{\pi_{i+1}}^{-1} \cdots x_{\pi_{p}}^{-1}}{\displaystyle \prod_{i\in [p]}(1-x_{\pi_i}\cdots x_{\pi_p})} \displaystyle \prod_{i\in [p]}x_{\pi_i}\cdots x_{\pi_p} \\
&=(-1)^p\frac {y_1^{s(1)}\cdots y_p^{s(p)}}{x_1\cdots x_p} \sum_{\tau \in \LL(P,s)} (\yy^*)^{-(r^*+\mathbf{1})}\frac {\displaystyle \prod_{i\in D_3(\tau^*)}  x_{\pi_{i+1}}^{-1} \cdots x_{\pi_{p}}^{-1}}{\displaystyle \prod_{i\in [p]}(1-x_{\pi_i}^{-1}\cdots x_{\pi_p}^{-1})} \\
&= (-1)^p\frac {y_1^{s(1)}\cdots y_p^{s(p)}}{x_1\cdots x_p} G_{(P^*,s^*)}((\xx^*)^{-1},(\yy^*)^{-1}),
\end{align*}
from which the theorem follows. 
\end{proof}

\begin{remark}
Theorem \ref{reci} generalizes the reciprocity theorem in \cite{BME2} which follows as the special case when $P$ is a naturally labeled chain.
\end{remark}

\section{Sign-ranked posets}\label{SR}
Let $P = \{1\prec 2 \prec \cdots \prec p\}$ be a naturally labeled chain, and let $s(i)=i$ for all $i \in [p]$. Savage and Schuster \cite[Lemma 1]{SS} proved that $A_{(P,s)}(t)$ is equal to the Eulerian polynomial 
$$
A_p(t)= \sum_{\pi \in \sym_p}t^{\des(\pi)},
$$
where $\des(\pi)= |\{ i \in [p] : \pi_i > \pi_{i+1}\}$. Recall that a polynomial $g(t)$ is \emph{palindromic} if $t^Ng(1/t)=g(t)$ for some integer $N$. It is well known that $A_p(t)$ is palindromic (in fact $t^{p-1}A_p(1/t)= A_p(t)$). The same is known to be true for the $P$-Eulerian polynomial of any naturally labeled graded poset, see \cite[Corollary 3.15.18]{EC1}, and more generally  for $P$-Eulerian polynomials of so called sign-graded labeled posets \cite[Corollary 2.4]{sign}. We shall here generalize these results to $(P,s)$-Eulerian polynomials. 

Recall that a pair of elements elements $(x,y)$ taken from a labeled poset $P$ is a \emph{covering relation} if $x \prec y$ and $x\prec z \prec y$ for no $z \in P$. Let $\EE(P)$ denote the set of covering relations of $P$. If $P$ is a labeled poset define a function $\epsilon : \EE(P) \rightarrow \{-1,1\}$ by 
$$
\epsilon(x,y) = \begin{cases}
1, &\mbox{ if } x<y, \mbox{ and }\\
-1, &\mbox{ if } x>y.
\end{cases}
$$
Sign-graded (labeled) posets, introduced in \cite{sign}, generalize graded naturally labeled posets. A labeled poset $P$ is \emph{sign-graded} of \emph{rank} $r$, if 
$$
\sum_{i=1}^k \epsilon(x_{i-1}, x_i) = r
$$
for each maximal chain $x_0\prec x_1 \prec \cdots \prec x_k$ in $P$. A sign-graded poset is equipped with a well-defined \emph{rank-function}, $\rho : P \rightarrow \ZZ$, defined by 
$$
\rho(x) = \sum_{i=1}^k \epsilon(x_{i-1}, x_i), 
$$
where $x_0\prec x_1 \prec \cdots \prec x_k=x$ is any unrefinable chain, $x_0$ is a minimal element and $x_k=x$. Hence a naturally labeled poset is sign-graded if and only if it is graded. 
A labeled poset $P$ is \emph{sign-ranked} if for each maximal element $x \in P$, the subposet 
$\{y \in P: y\preceq x\}$ is sign-graded. Note that each sign-ranked poset has a well-defined rank function $\rho : P \rightarrow \ZZ$. Thus a naturally labeled poset is sign-ranked if and only if it is ranked. 
\begin{theorem}\label{bij}
Let $P$ be a sign-ranked labeled poset and suppose its rank function attains non-negative values only.  Let $s(x)= \rho(x)+1$ for each $x \in [p]$, and define 
$u : \NN(P,s) \rightarrow \ZZ^{p}$ by $u(f)(x^*)= f(x)+\rho(x)$. Then 
$u : \NN_{\leq n}(P,s) \rightarrow \NN_{<n+1}(P^*,s^*)$ is a bijection for each $n \in \NN$. 
\end{theorem}

\begin{proof}
 We first prove  $u : \NN(P,s) \rightarrow \NN(P^*,s^*)$. Note that $f$ is a $(P,s)$-partition if and only if 
 \begin{enumerate}
 \item if $(x,y) \in \EE(P)$, then $f(x)/s(x) \leq f(y)/s(y)$, and 
 \item if $(x,y) \in \EE(P)$ and $\epsilon(x,y)=-1$, then $f(x)/s(x) < f(y)/s(y)$.
 \end{enumerate}
 Hence it suffices to consider covering relations when proving that $u : \NN(P,s) \rightarrow \NN(P^*,s^*)$. 

Let $f  \in \NN(P,s)$. 
Suppose $y$ covers $x$ and $\epsilon(x,y)=1$. Then $f(x)/s(x) \leq f(y)/s(y)$  and $s(x)<s(y)$, and thus 
$$
\frac {u(f)(x^*)}{s^*(x^*)}= \frac {f(x)+s(x)-1}{s(x)} \leq \frac {f(y)}{s(y)} + 1- \frac {1}{s(x)} < \frac {f(y)}{s(y)} + 1- \frac {1}{s(y)} = \frac {u(f)(y^*)}{s^*(y^*)},
$$
as desired. 

Suppose $y$ covers $x$ and $\epsilon(x,y)=-1$. Then $f(x)/s(x) < f(y)/s(y)$  and $s(x)=s(y)+1$ so that 
$$
\frac {u(f)(y^*)}{s^*(y^*)} -\frac {u(f)(x^*)}{s^*(x^*)} = \frac {f(y)}{s(y)} -\frac {f(x)}{s(y)+1}- \left( \frac {1}{s(y)} - \frac {1}{s(y)+1} \right).
$$
We want to prove that the quantity on either side of the equality above is nonnegative. 
By assumption 
$$
\frac {f(y)}{s(y)} -\frac {f(x)}{s(y)+1} = \frac {(s(y)+1)f(y)-s(y)f(x)}{s(y)(s(y)+1)} >0. 
$$
Hence $(s(y)+1)f(y)-s(y)f(x)$ is a positive integer, so that 
$$
\frac {f(y)}{s(y)} -\frac {f(x)}{s(y)+1} \geq \frac {1}{s(y)(s(y)+1)},
$$
as desired. Note that $u(f)$ is nonnegative since it is increasing and $u(f)(x^*) = f(x)$ when $x^*$ is a minimal element in $P^*$. Hence $u(f) \in \NN(P^*,s^*)$. 
 
 Let $\eta  : \NN(P^*,s^*) \rightarrow \ZZ^P$ be defined by 
 $\eta(g)(x)= g(x^*)-\rho(x)=g(x^*)+\rho^*(x^*)$, where $\rho^*$ is the rank function of $P^*$.  Clearly  $\eta  : \NN(P^*,s^*) \rightarrow \NN(P,s)$ by the exact same arguments as above.  Thus $u^{-1}=\eta$ and $u :  \NN(P,s) \rightarrow \NN(P^*,s^*)$ is a bijection.

Now 
 $u(f)(x^*)/s^*(x^*)= f(x)/s(x)+ (s(x)-1)/s(x)<n+1$ if $f \in \NN_{\leq n}(P,s)$ and $x \in P$, so that $u  : \NN_{\leq n}(P,s) \rightarrow \NN_{< n+1}(P^*,s^*)$ for each $n \in \NN$. 
 
 On the other hand if $g \in \NN_{< n+1}(P^*,s^*)$, then $g(x^*)=q(x^*)(\rho(x)+1) +r(x^*)$ where $0\leq q(x^*) \leq n$ and $0\leq r(x^*) \leq \rho(x)$. Hence 
$$
\frac {\eta(g)(x)}{s(x)} = \frac {g(x^*)}{\rho(x)+1} - \frac {\rho(x)} {\rho(x)+1} \leq  n+ \frac {r(x^*)} {\rho(x)+1}  -\frac {\rho(x)} {\rho(x)+1} \leq n.  
$$
Thus $\eta :  \NN_{< n+1}(P^*,s^*) \rightarrow \NN_{\leq n}(P,s)$ which proves the theorem. 
 \end{proof}

\begin{theorem}\label{recipr}
If $P$ is a sign-ranked labeled poset with nonnegative rank function $\rho$ and $s=\rho+1$, then 
$$
A_{(P,s)}(t)= t^{p-1}A_{(P,s)}(t^{-1})
$$
and 
$$
(-1)^p i(O(P,s),-t)= i(O(P,s),t-2).
$$
\end{theorem}
 
\begin{proof}
By  \eqref{r1}, \eqref{r2} and Theorem \ref{bij} 
$$
A_{(P,s)}(t)= \sum_{\tau \in \LL(P,s)} t^{|D(\tau)|}= \sum_{\tau^* \in \LL(P^*,s^*)} t^{|D_1(\tau^*)|}.
$$
The first part of the theorem now follows from \eqref{eul3} and \eqref{comp}. The second part follows from e.g., \cite[Lemma 3.15.11]{EC1}. 
\end{proof} 
 
\section{Real-rootedness and unimodality}\label{roots}
The Neggers-Stanley conjecture asserted that for each labeled poset $P$, the Eulerian polynomial $A_P(t)$ is real-rooted. Although the conjecture is refuted in its full generality \cite{BCount,SCount}, it is known to hold for certain classes of posets \cite{Brenti,Wa1}. Moreover, when $P$ is sign-graded, then the coefficients of $A_P(t)$ form a unimodal sequence \cite{sign,RW}. It is natural to ask for which pairs $(P,s)$
\begin{itemize}
\item[(a)] is $A_{(P,s)}(t)$ real-rooted? 
\item[(b)] do the coefficients of $A_{(P,s)}(t)$ form a unimodal sequence?
\end{itemize}
We first address (a). Suppose $P=([p], \preceq_P)$, $Q=([q], \preceq_Q)$  and $R=([p+q], \preceq_R)$ are labeled posets such that $[p+q]$ is the disjoint union of the two sets $\{u_1<u_2<\cdots < u_p\}$ and $\{v_1<v_2<\cdots < v_q\}$, and $x \preceq_R y$ if and only if either
\begin{itemize}
\item $x=u_i$ and $y = u_j$ for some $i,j \in [p]$ with $i \preceq_P j$, or 
\item $x=v_i$ and $y = v_j$ for some $i,j \in [q]$ with $i \preceq_Q j$.
\end{itemize}
We say that $R$ is a \emph{disjoint union} of $P$ and $Q$ and write $R= P\sqcup Q$. Moreover if $s_P : [p] \rightarrow \ZZ_+$ and 
$s_Q : [q] \rightarrow \ZZ_+$, then we define $s_{P\sqcup Q} : [p+q] \rightarrow \ZZ_+$ as the unique function satisfying 
$s_{P\sqcup Q}(u_i) = s_P(i)$ and $s_{P\sqcup Q}(v_j) = s_Q(j)$. 
\begin{proposition}
If the polynomials $A_{(P,s_P)}(t)$ and $A_{(Q,s_Q)}(t)$ are real-rooted, then so is the polynomial $A_{(P \sqcup Q, s_P\sqcup s_Q)}(t)$.
\end{proposition}
\begin{proof}
Clearly $$i((P \sqcup Q, s_P\sqcup s_Q),t)=i(O(P,s_P),t)\cdot i(O(Q,s_Q),t),$$
so the proposition follows from \cite[Theorem 0.1]{Wa2}. 
\end{proof}

It was proved in \cite{SaVis} that if $P=\{1\prec 2 \prec \cdots \prec p\}$ and $s : [p]\rightarrow \ZZ_+$ is arbitrary, then $A_{(P,s)}(t)$ is real-rooted. In Theorem \ref{ordinal} below we generalize this result to ordinal sums of anti-chains. If $P=(X, \preceq_P)$ and $Q=(Y, \preceq_Q)$ are posets on disjoint ground sets, then the \emph{ordinal sum}, $P\oplus Q= (X\cup Y, \preceq)$, is the poset with relations 
\begin{enumerate}
\item $x_1 \prec x_2$, for all $x_1,x_2 \in X$ with $x_1\prec_P x_2$, 
\item $y_1 \prec y_2$, for all  $y_1,y_2 \in X$ with $y_1\prec_Q y_2$, and 
\item $x \prec y$ for all $x \in X$ and $y \in Y$. 
\end{enumerate}

Let $f$ and $g$ be two real-rooted polynomials in $\RR[t]$ with positive leading coefficients. Let further $\alpha_1 \geq \alpha_2 \geq \cdots \geq \alpha_n$ and $\beta_1\geq \beta_2 \geq \cdots \geq \beta_m$ 
be the zeros of $f$ and $g$, respectively. If 
$$\cdots \leq \alpha_2 \leq \beta_2 \leq \alpha_1 \leq \beta_1$$
we say that $f$ is an \emph{interleaver} of $g$ and we write $f \ll g$. We also let $f\ll 0$ and $0\ll f$.
 We call a sequence $F_n=(f_i)^n_{i=1}$ of real-rooted polynomials \emph{interlacing} if $f_i \ll f_j$ for all $1\leq i < j \leq n$. We denote by $\mathcal{F}_n $ the family of all interlacing sequences $(f_i)_{i=1}^n$ of polynomials and we let $\mathcal{F}^+_n$ be the family of $(f_i)_{i=1}^n \in \mathcal{F}_n $ such that $f_i$ has nonnegative coefficients for all $1\leq i \leq n$. 

To avoid unnecessary technicalities we here redefine a labeled poset to be a poset $P=(S, \preceq)$, where $S$ is any set of positive integers. Thus $\LL(P)$ is now the set of rearrangements of $S$ that are also linear extensions of $P$. 

Equip $X(P,s):=\{ (k, x) : x \in P \mbox{ and } 0\leq k <s(x)\}$ with a total order defined by 
$(k, x) < (\ell, y)$ if $k/s(x) < \ell/s(y)$,  or $k/s(x) = \ell/s(y)$ and $x<y$. For $\gamma \in X(P,s)$, let 
$$
A_{(P,s)}^\gamma(t) = \mathop{\sum_{\tau=(\pi,r) \in \LL(P, s)}}_{(r(\pi_1), \pi_1)=\gamma} t^{\des_s(\tau)}.
$$
\begin{theorem}\label{ordinal}
Suppose $P=A_{p_1}\oplus \cdots \oplus A_{p_m}$ is an ordinal sum of anti-chains, and let $s : P \rightarrow \ZZ_+$ be a function which is constant on $A_{p_i}$ for $1\leq i \leq m$. Then $\{ A_{(P,s)}^\gamma(t) \}_{\gamma \in X}$, where $X=X(P,s)$, is an interlacing sequence of polynomials. 

In particular $A_{(P,s)}(t)$ and $A_{(P,s)}^\gamma(t)$ are real-rooted for all $\gamma \in X$. 
\end{theorem}

\begin{proof}
The proof is by induction over $m$. Suppose $m=1$, $p_1=n$, $A_n$ is the anti-chain on $[n]$, and $s(A_n)=\{s\}$. We prove the case $m=1$ by induction over $n$. If $n=1$ we get the sequence $1,t,t,\ldots, t$ which is interlacing. Otherwise if $\gamma=(k,\pi_1)$, then
\begin{equation*}
A_{(A_n,s)}^\gamma(t) = \sum_{\kappa < \gamma} tA_{(A_{n-1},s')}^\kappa(t)+ \sum_{\kappa \geq \gamma} A_{(A_{n-1},s')}^\kappa(t),
\end{equation*}
where $s'$ is $s$ restricted to $A_{n-1}$. This recursion preserves the interlacing property, see \cite[Theorem 2.3]{SaVis} and \cite{PSurvey}, which proves the case $m=1$ by induction. 

Suppose $m>1$. The proof for $m$ is again by induction over $p_1=n$. If $p_1=1$, then 
$$
A_{(P,s)}^\gamma(t) = \sum_{\kappa < \gamma} tA_{(P',s')}^\kappa(t)+ \sum_{\kappa > \gamma} A_{(P',s')}^\kappa(t),
$$
Where $P'=A_2 \oplus \cdots \oplus A_m$, and where  $s'$ is the restrictions to $P'$. Hence the case $p_1=1$ follows by induction (over $m$) since this recursion preserves the interlacing property, see \cite[Theorem 2.3]{SaVis}. 

The case $m>1$ and $p_1>1$ follows by induction over $p_1$ just as for the case $m=1$, $n>1$. 

Hence $\{A_{(P,s)}^\gamma(t)\}_\gamma$ is an interlacing sequence, and thus  
$$
A_{(P,s)}(t)= \sum_\gamma A_{(P,s)}^\gamma(t), 
$$
is real-rooted by e.g., \cite[Theorem 2.3]{SaVis}.
\end{proof}

Next we address (b). A palindromic polynomial $g(t)=a_0+a_1t +\cdots+a_nt^n$ may be written uniquely as 
$$
g(t) = \sum_{k=0}^{\lfloor d/2 \rfloor} \gamma_k(g) t^k (1+t)^{d-2k}, 
$$
where $\{\gamma_k(g)\}_{k=0}^{\lfloor d/2 \rfloor}$ are real numbers.
If $\gamma_k(g) \geq 0$ for all $k$, then we say that $g(t)$ is $\gamma$-\emph{positive}, see \cite{PSurvey}. Note that if $g(t)$ is $\gamma$-positive, then  $\{a_i\}_{i=0}^n$ is a \emph{unimodal sequence}, i.e., there is an index $m$ such that $a_0 \leq \cdots \leq a_m \geq a_{m+1} \geq \cdots \geq a_n$. 

\begin{conjecture}\label{gamma}
Suppose $P$ is a sign-ranked labeled poset with nonnegative rank function $\rho$ and $s=\rho+1$, then $A_{(P,s)}(t)$ is $\gamma$-positive. 
\end{conjecture}
\begin{remark}
Let $P$ be a sign-ranked labeled poset with a rank function $\rho$ with values only in $\{0,1\}$, and let $s=\rho+1$. Following the proof of \cite[Theorem 4.2]{sign}, with the use of Theorem \ref{ordinal},  it follows that Conjecture \ref{gamma} holds for $(P,s)$. We omit the technical details in recalling the proof here. 
\end{remark}
If $P$ is a naturally labeled ranked poset and $s=\rho+1$, then $O(P,s)$ is a closed integral polytope and $A_{(P,s)}(t)$ is the so called $h^*$-\emph{polynomial} of $O(P,s)$. If the following conjecture is true, then the coefficients of $A_{(P,s)}(t)$ form a unimodal sequence by a powerful theorem of Bruns and R\"omer \cite[Theorem 1]{BR}. 
\begin{conjecture}
Suppose $P$ is a naturally labeled ranked poset, and let $s=\rho+1$. Then $O(P,s)$ (or some related polytope with the same Ehrhart polynomial) has a regular and unimodular triangulation. 
\end{conjecture}

\section{Applications}\label{appl}
In this section we derive some applications of the generating function identities obtained in Section \ref{gener}. 
If $\alpha = (\alpha_1,\ldots, \alpha_p)$ is  a sequence, let $|\alpha|=\alpha_1+\cdots+\alpha_p$. For $\tau=(\pi ,r) \in \LL(P,s)$, let 
\begin{align*}
\comaj (\tau) &= \sum_{i\in D(\tau)} p-i, \mbox{ and }\\
\lhp (\tau) &=|r| + \sum_{i\in D(\tau)} s(\pi_{i+1}) + \cdots + s(\pi_p) 
\end{align*}
\begin{theorem}
\label{uq}
If $P$ is a labeled poset and $s : [p] \rightarrow \ZZ_+$, then 
\begin{eqnarray}
\sum_{n \geq 0} \left( \sum_{f \in \NN_{\leq n}(P,s)} q^{|r(f)|}u^{|q(f)|} \right)  t^n
 = \frac{\displaystyle  \sum_{\tau \in \LL(P,s)} q^{|r|} u^{\comaj (\tau)}t^{\des_s (\tau)} }
 {\displaystyle \prod_{i=0}^p (1-u^it)}. 
 \end{eqnarray}

\end{theorem}

\begin{proof}
Set $x_i= u$ and $y_i=q$ for all $1\leq i \leq p$ in (\ref{r1}).
Then
\begin{eqnarray*}
\sum_{\tau \in \LL(P,s)} \yy^r\frac {\displaystyle \prod_{i\in D(\tau)}  x_{\pi_{i+1}} \cdots x_{\pi_{p}}}{\displaystyle \prod_{i\in [p]}(1-x_{\pi_i}\cdots x_{\pi_p}t)}\frac {t^{|D(\tau)|}} {1-t} &=&
\sum_{\tau \in \LL(P,s)} \frac{\displaystyle  q^{|r|} u^{\comaj (\tau)}t^{\des_s(\tau)} }{\displaystyle \prod_{i\in [p]}(1-tu^{p+1-i})(1-t)}\\
&=&
 \frac{\displaystyle \sum_{\tau \in \LL(P,s)} q^{|r|} u^{\comaj (\tau)}t^{\des_s(\tau)} }{\displaystyle  \prod_{i\in [p]}(1-tu^{i})(1-t)}.
\end{eqnarray*}
The theorem follows.
\end{proof}

\begin{theorem}
If $P$ is a labeled poset and $s : [p] \rightarrow \ZZ_+$, then 
\begin{eqnarray}
\sum_{n \geq 0} \left( \sum_{f \in \NN_{\leq n}(P,s)} q^{|f|}\right) t^n
 = \sum_{\tau \in \LL(P,s)} \frac{\displaystyle  q^{\lhp (\tau)}t^{\des_s(\tau)} }
 {\displaystyle  \prod_{i\in [p]} \left(1-tq^{\sum_{j=i}^p s(\pi_j)} \right) (1-t)  }.
 \end{eqnarray}
\end{theorem}

\begin{proof}
Set $x_i= q^{s(i)}$ and $y_i=q$ for all $1\leq i \leq p$ in (\ref{r1}).
\end{proof}

\begin{corollary} 
\label{qv}
If $P$ is an anti-chain and $s : [p] \rightarrow \ZZ_+$, then 
\begin{eqnarray}
\label{qve}
\sum_{n \geq 0} \prod_{i=1}^p   \left(   u^n + [n]_{u}[s(i)]_q       \right) t^n
 =  \frac{\displaystyle \sum_{\tau \in \LL(P,s)} q^{|r|} u^{\comaj (\tau)}t^{\des_s(\tau)} }
 {\displaystyle \prod_{i=0}^p (1-u^it)}.
 \end{eqnarray}
\end{corollary}
\begin{proof}
The corollary follows from Theorem \ref{uq} and Corollary \ref{6}.
\end{proof}
 The \emph{wreath product} of $\sym_p$ with a cyclic group of order $k$ has elements 
$$\ZZ_k \wr \sym_p = \{(\pi, r) :\pi\in \sym_p \text{ and }  r : [p] \rightarrow \ZZ_k\}.$$
The elements of $\ZZ_k \wr \sym_p$ are often thought of as $r$-colored permutations. We may identify $\ZZ_k \wr \sym_p$ with 
$\LL(P,s)$ where $P$ is an anti-chain on $[p]$ and $s(i)=k$ for all $k \in [p]$. 
For $\tau=(\pi ,r) \in  \ZZ_k \wr \sym_p$ define 
\begin{align*}
\fmaj (\tau) &= |r| + k\cdot \comaj (\tau).
\end{align*}
Note that $\lhp (\tau)$ agrees with $\fmaj (\tau)$ when $s=(k,k, \ldots  , k).$

Below we derive a Carlitz formula for  $\ZZ_k \wr \sym_p$ first proved by Chow and Mansour in \cite{ChMa}. 
\begin{corollary} 
\label{kn1}
For positive integers $p$ and $k$,
\begin{eqnarray}
\label{ant}
\sum_{n\geq 0}  [kn+1]_q^p t^n
 =  \frac{ \displaystyle  \sum_{\tau \in \ZZ_k \wr \sym_p}   t^{\des_s(\tau)}  q^{\fmaj ( \tau )}     }   {\displaystyle \prod_{i=0}^{p} \left( 1-t q^{ki}   \right)}. 
\end{eqnarray}
\end{corollary}
\begin{proof}
Let $s=(k, k, \ldots , k)$ and set $u=q^k $ in (\ref{qve}). Then 
\begin{eqnarray*}
 \prod_{i=1}^p   \left(   u^n + [n]_{u}[s(i)]_q       \right) &=&  \left(   q^{nk} + [n]_{q^k}[k]_q       \right)^p \\
 &=& \left( q^{nk}    + \frac{q^{kn}-1}{q^k-1} \frac{q^k-1}{q-1}          \right)^p\\
 &=& [nk+1]_q^p.
\end{eqnarray*}
The right hand side follows since $s(i)=k$ for all $1\leq i \leq p$, and thus we sum over all $\tau \in \ZZ_k \wr \sym_p.$ 
\end{proof}

\begin{remark}
The definition of $\fmaj$ above differs from the definition of the flag major index $\fmaj_r$ in \cite{ChMa}.
By the change in variables $q \rightarrow q^{-1}$ and $t \rightarrow tq^{kp}$ and by noting that
$[kn+1]_q^p t^n $ is invariant under this change of variables we find that the two flag major indices have the same distribution. 
\end{remark}

\begin{corollary} 
\label{kn}
For positive integers $p$ and $k$, 
$$
\sum_{n\geq 0}  \prod_{i=1}^p(1+n[k]_{q_i}) t^n
 =  \frac{\displaystyle   \sum_{\tau \in \ZZ_k \wr \sym_p} q_1^{r_1}q_2^{r_2}\cdots q_p^{r_p}  t^{\des_s (\tau)}   }  {(1-t)^{p+1}}.  
$$

\end{corollary}
\begin{proof}
Let $s=(k, k, \ldots , k)$ and set $x_i=1 $ for all $1\leq i \leq p$ in the equation displayed in Corollary \ref{6}. 
\end{proof}

\begin{remark}
Note that when $q_i\geq 0$ for all $1 \leq i \leq p$, the polynomial $$n \mapsto \prod_{i=1}^p(1+n[k]_{q_i})$$ has all its zeros in the interval $[-1,0]$. By an application of 
\cite[Theorem 0.1]{Wa2} it follows that the polynomial 
$$
  \sum_{\tau \in \ZZ_k \wr \sym_p} q_1^{r_1}q_2^{r_2}\cdots q_p^{r_p}   t^{\des_s(\tau)}  
$$
is real-rooted in $t$. This generalizes  \cite[Theorem 6.4]{B2004}, where the case $k=2$ was obtained.
\end{remark}

\end{document}